\documentclass[12pt]{elsarticle}
\usepackage{amssymb,latexsym,amsmath}
\usepackage{graphicx,times}
\usepackage{amsthm}
\usepackage{dcolumn}
\usepackage[margin=1.0in]{geometry}
\newtheorem{thm}{Theorem}

\newtheorem{cor}[thm]{Corollary}
\newtheorem{Def}{Definition}[section]

\numberwithin{equation}{section}
\journal{Journal of \LaTeX\ Templates}

\begin{document}

\begin{frontmatter}

\title{ Existence and Uniqueness Theorems for Differential Equations with Proportional Delay}


\author{Prajakta Rajmane$^{a*}$ }
\ead{prajaktarajmane@gmail.com}
\author{Jayvant Patade$^{b}$ }
\ead{dr.jayvantpatade@gmail.com}
\author{M. T. Gophane$^{a}$ }
\ead{mtg\_math@unishivaji.ac.in}
\address{$^{a}$Department of Mathematics, Shivaji University, Kolhapur - 416004, India.\\ $^{b}$Department of Mathematics, Jaysingpur College, Jaysingpur - 416101, India.}


\cortext[mycorrespondingauthor]{Corresponding author}
\begin{abstract}
The differential equation (DE) with proportional delay is a particular case of the time-dependent delay differential equation (DDE). In this paper, we solve non-linear  DEs with proportional delay using the successive approximation method (SAM). We prove the existence, uniqueness of theorems, and stability for DEs with proportional delay using SAM. We derive convergence results for these equations by using the Lipschitz condition. We generalize these results to the fractional differential equations (FDEs) and system of FDEs containing Caputo fractional derivative. Further, we obtain the series solution of the pantograph equation and Ambartsumian equation in the form of a power series which are convergent for all reals. Finally, we illustrate the efficacy of the SAM  by example. The results obtained by SAM are compared with exact solutions and other iterative methods. It is observed that SAM is simpler compared to other methods and the solutions obtained using SAM are consistent with the exact solution.

\end{abstract}
\end{frontmatter}

\section{Introduction}
The delay differential equations (DDE) contain the state variable term at a past time $t-\tau$. The inclusion of the delay $\tau$ makes the DDE an infinite dimensional dynamical system. Even if it is very difficult to analyze and solve such equations, this branch is popular among the applied scientists due to the applications in various fields. 
\par On the other hand, if the order of the derivative in a differential equation is any arbitrary number (instead of a positive integer) then the equation is called as the fractional differential equation (FDE).  Even though there are several inequivalent definitions of fractional derivative operator, one can select the derivative which is appropriate for the model under consideration. This flexibility is a key feature behind the popularity of fractional calculus.
\par Daftardar-Gejji and coworkers proposed numerical schemes \cite{bhalekar2011predictor, daftardar2015solving} for solving fractional order delay differential equations (FDDE).  Modified Laguerre wavelets method \cite{iqbal2015modified}, spectral collocation method \cite{ali2019chebyshev},  fractional-order fibonacci-hybrid functions \cite{sabermahani2020fractional} are few other methods for solving FDDEs.  Stability analysis of FDDEs is proposed in \cite{bhalekar2013stability,bhalekar2016stability,kaslik2012analytical,bhalekar2019analysis}. Applications of FDDE are presented in \cite{bhalekar2011fractional,rihan2020dynamics,latha2017fractional,rihan2021fractional}.

\par In general, the delay $\tau$ in the DDE ${x'(t)}=f(t,x(t),x(t-\tau))$ is not constant. The analysis becomes more difficult when $\tau$ depends on time or state. The proportional delay differential equation ${x'(t)}=f(t,x(t),x(qt))$ or a pantograph equation is a particular case of time-dependent DDE with $\tau(t)=(1-q)t$. 
These equations are proposed by Ockendon and Tayler  in the seminal work \cite{ockendon1971dynamics} to model the motion of an overhead trolley wire. Few other applications of these equations are discussed in \cite{buhmann1993stability,bhalekar2017series}. The Daftardar-Gejji and Jafari method (DJM) is applied in \cite{bhalekar2016analytical} to find analytical solutions of pantograph equation. Further, the authors presented the various relations of the solution series with the existing special functions.  Patade and Bhalekar  proposed the power series solution Ambartsumian equation \cite{patade2017analytical} by using DJM. The analytical solution of pantograph equation are discussed in \cite{patade2017analytical1}. 
 In this paper, we generalize the results in \cite{coddington2012introduction} on the existence-uniqueness of ordinary differential equations (ODE) to differential equations with proportional delay,  FDEs proportional delay and system of FDEs with proportional delay. We use the successive approximation method (SAM) to prove our results.
 
  The paper is organized as follows. In Section 2, we
   give definitions and notations of fractional derivatives and integrals. The differential equations with proportional delay are described in Section 3. 
   Successive approximate solutions and existence theorem are discussed in Section 4. The stability analysis is presented in Section 5. The series solution of the pantograph equation and Ambartsumian equation are described in Section 6.   
     The generalization of these results to  FDEs and the system of FDEs are derived in Section 7 and Section 8. Section 9 deals with illustrative
     example and the conclusions are summarized in Section 10. 
\section{Preliminaries and Notations} \label{Preliminaries}
\begin{Def}\cite{kilbas2006theory}
	The  Riemann-Liouville fractional integral of order $\alpha> 0 $ of $f\in C[0, \infty) $ is defined as
	\begin{equation}
	I^\alpha f(t)=\frac{1}{\Gamma(\alpha)} \int_{0}^t (t-\zeta)^{\alpha-1}f(\zeta)d\zeta,\quad t>0.
	\end{equation}
\end{Def}
\begin{Def}\cite{kilbas2006theory}
	The (left sided) Caputo fractional derivative of $f, f \in C_{-1}^m, m\in\mathbb{N}\cup\{0\}$, is defined as:
	\begin{eqnarray}
	D^\alpha f(t)&=&\frac{d^m}{ dt^m} f(t),\quad \alpha = m \nonumber\\
	&=& I^{m-\alpha}\frac{d^m}{ dt^m} f(t),\quad {m-1} <\alpha <m,\quad m\in \mathbb{N}.
	\end{eqnarray}
\end{Def}
Note that for $0\le m-1 < \alpha \le m$ and $\beta>-1$
\begin{eqnarray}
I^\alpha x^\beta &=&\frac{\Gamma{(\beta+1)}}{ \Gamma{(\beta+\alpha+1)}} x^{\beta+\alpha},\nonumber\\
\left(I^\alpha D^\alpha f\right)(t)&=& f(t)-\sum_{k=0}^{m-1} f^{(k)}(0)\frac{t^k}{k!}.
\end{eqnarray}
\begin{Def}\cite{kilbas2006theory}
	The  Mittag-Leffler function is defined as  
	\begin{equation}
	E_\alpha (t)=\sum_{n=0}^\infty\frac{t^n}{\Gamma{(\alpha n+1)}}, \quad \alpha>0.
	\end{equation}
\end{Def}

\begin{Def}\cite{kilbas2006theory}
 The multi-parameter Mittag-Leffler function is defined as:
	\begin{equation*}
	E_{(\alpha_{1},\cdots,\alpha_{n}), \beta}(z_1, z_2,\cdots,z_{n})=\sum^\infty_{k=0}\mathop{\sum_{l_1+\cdots+l_{n } =k}}_{l_j\geq 0}(k;l_1,\cdots,l_{n})\left[\frac{\displaystyle\prod_{j=1}^{n }z^{l_j}_{j}}{\Gamma(\beta +\displaystyle\sum^{n}_{j=1}\alpha_{j}l_j)}\right].
	\end{equation*}
	where, $\left(k;l_1,l_2,\cdots,l_n\right)$ is the multinomial coefficient defined as
	\begin{equation}
	\left(k;l_1,l_2,\cdots,l_n\right)=\frac{k!}{l_1!l_2!\cdots l_n!}.
	\end{equation}
\end{Def}

\section{Differential Equations with Proportional delay}
Consider the differential equations with proportional delay
\begin{equation}
y'(t)=f(t, {y}(t), {y}(qt)),y(0)=y_0, 0<q<1, \label{1}
\end{equation}
where  $f$ is a continuous function defined on some rectangle
\begin{equation}
R=\{|t|\leq a,|y(t)-y_0|\leq b,|y(qt)-y_0|\leq b, a>0,\, b>0\}\nonumber.
\end{equation}

\begin{thm}
A function $\phi$ is a solution of the IVP (\ref{1}) on an interval $I$ if and only if it is a solution of the integral equation 
\begin{equation}
y(t)=y_0+\int_{0}^{t} f(x, {y}(x), {y}(qx))dx \quad \text{on}\quad I \label{2.1}
\end{equation}
\end{thm}
\begin{proof}
Let $\phi$ is a solution of the IVP (\ref{1}) on an interval $I$. Then\\
\begin{equation}
\phi'(t)=f(t, {\phi}(t), {\phi}(qt)),\phi(0)=y_0, 0<q<1 \label{3.1}
\end{equation}
The equivalent integral equation (\ref{3.1}) is
\begin{equation}
\phi(t)=\phi(0)+\int_{0}^{t} f(x, {\phi}(x), {\phi}(qx))dx.\label{4.1}
\end{equation}
and $ \phi(0)=y_0 $. Thus $\phi$ is a solution of the IVP (\ref{2.1}).\\
Conversely, suppose equation (\ref{4.1}) hold. Differentiate equation (\ref{4.1}) w.r.t. $t$, we get
\begin{equation}
\phi'(t)=f(t, {\phi}(t), {\phi}(qt)), 0<q<1 \quad \forall t\in I\nonumber
\end{equation}
From equation (\ref{2.1}) $\phi(0)=y_0$.\\
Hence $\phi$ is a solution of the IVP (\ref{1}).
\end{proof}

\section{Successive Approximate Solution for Differential Equations with Proportional Delay}

Let $\phi_0(t)=y_0$ be the first approximate solution of the IVP (\ref{1}). Then
\begin{eqnarray}
\phi_1(t)&=&y_0+\int_{0}^{t} f(x, {\phi_0}(x), {\phi_0}(qx))dx.\nonumber\\
\phi_2(t)&=&y_0+\int_{0}^{t} f(x, {\phi_1}(x), {\phi_1}(qx))dx.\nonumber
\end{eqnarray}
Continuing in this way, we obtain
\begin{equation}
\phi_{k+1}(t)=y_0+\int_{0}^{t} f(x, {\phi_k}(x), {\phi_k}(qx))dx.\quad k=0,1,2,\cdots.\label{22}
\end{equation}

\begin{thm}\label{thm2}
Let f is continuous and $|f|\leq M$ on $R$. The successive approximation (\ref{22}) exist and continuous on the interval $I=[-\zeta,\zeta]$, where $\zeta=\text{min}\left\{a,\frac{b}{M}\right\}$.  
	If $t\in I$  then $\left(t,y(t),y(qt)\right)\in R$ and $|\phi_{k}(t)-y_0|\leq M|t|$, $|\phi_{k}(qt)-y_0|\leq M|t|$.
\end{thm}
\begin{proof}
We prove the result by mathematical induction.\\
(i) Clearly  $\phi(0)=y_0$ is continuous on $I$. Thus, theorem is true for $k=0$.\\
(ii) For $k=1$, we have
\begin{eqnarray}
\phi_1(t)&=& y_0+\int_{0}^{t} f(x, {\phi_0}(x), {\phi_0}(qx))dx.\nonumber\\
\phi_1(t)&=& y_0+\int_{0}^{t} f(x, y_0, y_0)dx.\nonumber
\end{eqnarray}
Since $f$ is continuous and hence,   $\phi_1(t)$ exist.
 \begin{eqnarray}
|\phi_1(t)-y_0| &=&|\int_{0}^{t} f(x, {\phi_0}(x), {\phi_0}(qx))dx|.\nonumber\\
 &\leq&\int_{0}^{t}| f(x, {\phi_0}(x), {\phi_0}(qx))|dx.\nonumber\\
 &\leq& M|t| \nonumber\\
 &\leq& b,  \quad t\in I\nonumber\\
\text{and}\quad |\phi_1(qt)-y_0| &\leq& M|qt| \nonumber\\
&\leq& M|t|, \quad 0<q<1 \nonumber\\
 &\leq& b,  \quad t\in I\nonumber
\end{eqnarray}
Thus, for $t\in I$, $\left(t,y(t),y(qt)\right)\in R$ and $|\phi_{1}(t)-y_0|\leq M|t|$, $|\phi_{1}(qt)-y_0|\leq M|t|$.\\
The theorem is true for $k=1$\\
(iii)  Assume that theorem is true for $k=n$.\\
i.e. For $t\in I$, $\left(t,y(t), y(qt)\right)\in R$ and $|\phi_{n}(t)-y_0|\leq M|t|$, $|\phi_{n}(qt)-y_0|\leq M|t|$.\\
(iv) To prove the theorem for $k=n+1$.\\
If $t\in I$, then 
 \begin{eqnarray}
\phi_{n+1}(t)&=&y_0+\int_{0}^{t} f(x, {\phi_n}(x), {\phi_n}(qx))dx.\nonumber
\end{eqnarray}
Since $f$ is continuous and hence,   $\phi_{n+1}(t)$ exist on $I$.
\begin{eqnarray}
|\phi_{n+1}(t)-y_0|  &\leq& M|t| \nonumber\\
 &\leq& b,  \quad t\in I\nonumber\\
\text{and}\quad |\phi_{n+1}(qt)-y_0| &\leq& M|qt| \nonumber\\
&\leq& M|t|, \quad 0<q<1 \nonumber\\
 &\leq& b,  \quad t\in I\nonumber
\end{eqnarray}
Thus, if  $t\in I$, $\left(t,y(t),y(qt)\right)\in R$ and $|\phi_{n+1}(t)-y_0|\leq M|t|$, $|\phi_{n+1}(qt)-y_0|\leq M|t|$.\\
Hence by mathematical induction, the result is true for all positive integer $n$.
\end{proof}

\begin{thm}(Existence Theorem)
Let f is continuous and $|f|\leq M$ on the rectangle
\begin{equation}
R=\{|t|\leq a,|y(t)-y_0|\leq b,|y(qt)-y_0|\leq b, a>0,\, b>0\}\nonumber.
\end{equation}
Suppose $f$ satisfies Lipschitz condition in second and third variable with  Lipschitz constants $L_1$ and $L_2$ such that 
\begin{equation}
|f(t, {y_1}(t), {y_1}(qt))-f(t, {y_2}(t), {y_2}(qt))|\leq L_1|y_1(t)-y_2(t)|+L_1|y_1(qt)-y_2(qt)|\nonumber.
\end{equation}
Then the successive approximations (\ref{22}) converges on the interval  $I=[-\zeta,\zeta]$, where $\zeta=\text{min}\left\{a,\frac{b}{M}\right\}$ to a solution  $\phi$ of the IVP  (\ref{1}) on $I$.
\end{thm}
\begin{proof}
We have 
\begin{equation}
\phi_k(t)=\phi_0(t)+\sum_{n=1}^{k}[\phi_n(t)-\phi_{n-1}(t)]\nonumber.
\end{equation}
To prove the sequence $\{\phi_k\}$ converges, it is enough to prove the series 
\begin{equation}
\phi_0(t)+\sum_{n=1}^{\infty}[\phi_n(t)-\phi_{n-1}(t)]\label{6.1}
\end{equation}
is convergent.\\
By theorem (\ref{thm2}) the function $\phi_k$ all exist and continuous on $I$.\\
Also, $|\phi_{1}(t)-\phi_{0}(t)|\leq M|t|$ and $|\phi_{1}(qt)-\phi_{0}(qt)|\leq M|t|$ for $t\in I$.\\
Now,
\begin{eqnarray}
\phi_{2}(t)-\phi_{1} (t)&=&\int_{0}^{t}[ f(x, {\phi_1}(x), {\phi_1}(qx))- f(x, {\phi_0}(x), {\phi_0}(qx))]dx\nonumber\\
\therefore |\phi_{2}(t)-\phi_{1} (t)|&\leq& \int_{0}^{t}| f(x, {\phi_1}(x), {\phi_1}(qx))- f(x, {\phi_0}(x), {\phi_0}(qx))|dx\nonumber\\
&\leq& \int_{0}^{t}[ L_1 |{\phi_1}(x)-{\phi_0}(x)|+L_2 |{\phi_1}(qx)-  {\phi_0}(qx)|]dx\nonumber\\
&\leq& M(L_1+L_2)\frac{|t|^2}{2}\nonumber.
\end{eqnarray}
We shall prove by mathematical induction
\begin{equation}
|\phi_{n}(t)-\phi_{n-1} (t)|\leq M(L_1+L_2)^{n-1}\frac{|t|^n}{n!}\label{7}
\end{equation}
We have prove that equation (\ref{7}) true for $n=1,2.$\\
Assume that (\ref{7}) true for $n=m.$\\
We have 
\begin{eqnarray}
\phi_{m+1}(t)-\phi_{m} (t)&=&\int_{0}^{t}[ f(x, {\phi_m}(x), {\phi_m}(qx))- f(x, {\phi_{m-1}}(x), {\phi_{m-1}}(qx))]dx\nonumber\\
\therefore |\phi_{m+1}(t)-\phi_{m} (t)|&\leq& \int_{0}^{t}| f(x, {\phi_m}(x), {\phi_m}(qx))- f(x, {\phi_{m-1}}(x), {\phi_{m-1}}(qx)|dx\nonumber\\
&\leq& \int_{0}^{t}[ L_1 | {\phi_m}(x)-\phi_{m-1}(x)|+L_2 |{\phi_m}(qx)-  {\phi_{m-1}}(qx)|]dx\nonumber\\
&\leq& M(L_1+L_2)^{m}\frac{|t|^{m+1}}{(m+1)!}\label{6}.\nonumber
\end{eqnarray}
Thus, the result is true for $ n=m+1 $.\\
Hence, by the mathematical induction result is true for all $ n=1,2,\cdots.$\\
Therefore, the infinite series  (\ref{7}) is absolutely convergent on $I$.
This shows that the $n^{th}$ term of the series $|\phi_0(t)|+\sum_{n=1}^{\infty}|\phi_n(t)-\phi_{n-1}(t)|$ is less than $\frac{M}{(L_1+L_2)}$ times the $n^{th}$ term of the power series $e^{(L_1+L_2)|t|}$.
Hence The series (\ref{7}) is convergent.
\end{proof}

\section{Stability Analysis}
 The differential equations with proportional delay
\begin{equation}
y'(t) = f\left(t, y(t), y(qt)\right),
\end{equation}
is a special case of the time-dependent delay differential equation (DDE)
\begin{equation}
y'(t) = f\left(t, y(t), y\left(t-\tau(t)\right)\right)\quad \textrm{with}\quad\tau(t)=(1-q)t,  \nonumber
\end{equation}

\begin{Def}\cite{deng2006stability}
Consider the  DDE,
\begin{equation}
y'(t) = f(y(t), y(t-\tau(t))),\label{13}
\end{equation}
where $f:R\times R\rightarrow R$.
The flow $\phi_t(t_0)$ is a solution $y(t)$ of ({\ref{13}}) with initial condition $y(t) = t_0,\, t\leq 0$. The point $y^*$ is called equilibrium solution of  ({\ref{13}}) if $f(y^*,y^*)=0$.\\
\textbf{(a)} If, for any  $\epsilon > 0$, there exist  $\delta > 0$ such that $ |t_0-y^*|< \delta \Rightarrow |\phi_t(t_0)-y^*| < \epsilon,$ then the system ({\ref{13}}) is stable (in the Lyapunov sense) at the equilibrium  $y^*$.\\
\textbf{(b)} If the system ({\ref{13}}) is stable at   $y^*$ and moreover,$\lim\limits_{t\rightarrow\infty}|\phi_t(t_0)-y^*|=0$ then the system  ({\ref{13}}) is said to be asymptotically stable at $y^*$. 
\end{Def}
The following results are similar to those in \cite{deng2006stability}

\begin{thm}
Suppose that the  equilibrium solution $y^*$ of the equation
\begin{equation}
y' = f(y(t),y(t-\tau^*)),\quad \tau^*= \tau(t_0)
\end{equation}
is stable and 
$\|f(y(t),y(t-\tau(t))) - f(y(t),y(t-\tau(t_1)))\|<\epsilon_1 |t-t_1|$,
for some $\epsilon_1>0$ and $t,t_1\in [t_0, t_0+c),$
c is a positive constant, then there exists $\bar{t}> 0$ such that
the equilibrium solution $y^*$ of Eq. ({\ref{13}}) is stable on finite time interval $[t_0,\bar{t})$.
\end{thm}
\begin{cor}
If the real parts of all roots of $ \lambda -a-be^{-\lambda\tau^*}=0$ are negative, where $a =\partial_1f, b = \partial_2f$ evaluated at equilibrium. Then there exist $\epsilon_c, \bar{t}(>t_0) $, such that when $\epsilon_1 <\epsilon_c$, the solution  $y^*=0$ of Eq. ({\ref{13}}) is stable on finite time interval  $[t_0,\bar{t})$.
\end{cor}

\section{Series Solution of Pantograph Equation}
A pantograph is a device used in electric trains to collect current from overloaded lines. The pantograph equation was formulated by Ockendon and Taylor in 1971 and originates in electrodynamics.\\
Consider the pantograph equation,

\begin{equation}
y'(t) = ay(t)+ by(q t),\quad y(0)= 1, \label{14}
\end{equation}
where $0<q<1$, $a, b \in {R}$.

Integrating (\ref{14}), we get
\begin{equation}
y(t)= 1 + \int_{0}^{t}\left(ay(x)+ by(qx)\right)dt\label{15}
\end{equation}
Suppose $\phi_k(t)$ be the $k^{th}$ approximate solution, where the initial approximate solution is
taken as
\begin{equation}
\phi_0(t)=1.
\end{equation}

For $k \ge 1$, the recurrent formula as below:
\begin{equation}
\phi_k (t) =1 + \int_{0}^{t}\left(a\phi_{k-1}(x)+b\phi_{k-1}(qx)\right)dx\label{16}.
\end{equation}
From the recurrent formula, we have 
\begin{eqnarray}
\phi_1(t)&=& 1+\int_{0}^{t}\left( a\phi_0(x)+ by_0(q x)\right)dx\nonumber\\
&=& 1+ (a+b)\frac{t}{1!},\nonumber\\
\phi_2(t)&=& 1+ \int_{0}^{t}\left( a\phi_1(x)+ b\phi_1(q x)\right)dt\nonumber\\
&=& 1+(a+b)\frac{t}{1!}+(a+b)(a+bq)\frac{t^2}{2!},\nonumber\\
\phi_3(t)&=& 1 +\int_{0}^{t}\left(a\phi_2(t)+ b\phi_2(q t)\right)dt\nonumber\\
&=& 1+(a+b)\frac{t}{1!}+(a+b)(a+bq)\frac{t^2}{2!} (a+b)(a+bq)(a+bq^2)\frac{t^3}{3!},\nonumber\\
&\vdots& \nonumber\\
\phi_k(x)&=& 1+\frac{t^k}{k!}\prod_{j=0}^{k-1}\left(a+bq^j\right), \quad k=1,2,3\cdots. \nonumber\\
&&\text{As} \quad k\rightarrow\infty,\quad \phi_k (t)\rightarrow y(t)\nonumber\\
y(t) &=&  1 + \sum_{m=1}^\infty \frac{t^m}{m!}\prod_{j=0}^{m-1}\left(a+bq^j\right).\nonumber
\end{eqnarray}
If we define $\prod_{j=0}^{m-1}\left(a+bq^j\right)=1$, for $m=0$,
then
\begin{equation}
		y(t) = \sum_{m=0}^\infty \frac{t^m}{m!}\prod_{j=0}^{m-1}\left(a+bq^j\right).\label{17}
\end{equation}

\begin{thm}
	For $0<q<1$, the  power series (\ref{17})	is  convergent for $t\in{R}$.
\end{thm}

\begin{cor}
	The power series (\ref{17}) is  absolutely convergent for all $t$ and hence it is uniformly convergent on any compact interval on ${R}$.
\end{cor}

\begin{thm}
If $0<q<1$, $a, b\ge 0$, then
\begin{equation}
 e^{at}\le y(t) =\sum_{m=0}^\infty \frac{t^m}{m!}\prod_{j=0}^{m-1}\left(a+bq^j\right)\le e^{(a+b+c)t},\quad 0 \le t < \infty.\nonumber
\end{equation}
\end{thm}

\begin{thm}
If $(a+b)< 0$ then zero solution of ({\ref{14}}) is asymptotically stable.
\end{thm}
\begin{proof}
 \begin{eqnarray}
 \text{Define}\qquad u(t)&=&\max_{0\leq x \leq t}y^2(t)\nonumber\\
\therefore\frac{1}{2}u'(t)&=& \frac{1}{2} \frac{d}{dt}(y^2(t))\nonumber\\
&=& y(t)y'(t)\nonumber\\
&=& y(t)(ay(t)+ by(q t))\nonumber\\
&=& a y^2(t) + b y(t) y(qt)\nonumber\\
&\leq& (a+b)u(t)\nonumber\\
\Rightarrow u(t)&\leq& u(0)e^{2(a+b)t}\nonumber\\
\therefore \lim\limits_{x\rightarrow\infty}y(t)&=& 0, \quad \textrm{if} \quad(a+b)< 0.\nonumber
\end{eqnarray}
\end{proof}

\subsection{Series Solution of Ambartsumian Equation}
In \cite{ambartsumian1944fluctuation} Ambartsumian derived a delay differential equation describing the fluctuations of the surface brightness in a milky way. The equation is described as:
\begin{equation}
y'(t) = -y(t)+\frac{1}{q}y\left(\frac{t}{q}\right)\label{a1}
\end{equation}
where $q>1$ and is constant for the given model.\\
The Eq.(\ref{a1}) with initial condition $y(0)=\lambda$ can be written equivalently as
\begin{equation}
y(t)= \lambda + \int_{0}^{t}\left(\frac{1}{q}y\left(\frac{x}{q}\right)-y(x)\right)dx.\label{a2}
\end{equation}
Suppose $\phi_k(t)$ be the $k^{th}$ approximate solution, where the initial approximate solution is
taken as
\begin{equation}
\phi_0(t)=\lambda.
\end{equation}
For $k \ge 1$, the recurrent formula as below:
\begin{equation}
\phi_k (t) =\lambda + \int_{0}^{t}\left(\frac{1}{q}\phi_{k-1}\left(\frac{x}{q}\right)-\phi_{k-1}(x)\right)dx\label{a4}.
\end{equation}
From the recurrent formula, we have 
\begin{eqnarray}
\phi_1 (t) &=& \lambda + \int_{0}^{t}\left(\frac{1}{q}\phi_{0}\left(\frac{x}{q}\right)-\phi_{0}(x)\right)dx\nonumber\\
&=&\lambda +\int_{0}^{t}\left(\frac{\lambda}{q}-\lambda\right)dx\nonumber\\
&=&\lambda +\left(\frac{\lambda}{q}-\lambda\right)\frac{t}{1!}\nonumber\\
&=&\left(1+\left(\frac{1}{q}-1\right)\frac{t}{1!}\right)\lambda, \nonumber\\
\phi_2 (t)&=&\lambda + \int_{0}^{t}\left(\frac{1}{q}\phi_{1}\left(\frac{x}{q}\right)-\phi_{1}(x)\right)dx\nonumber\\
&=&  \left(1+\left(\frac{1}{q}-1\right)\frac{t}{1!}+\left(\frac{1}{q}-1\right)\left(\frac{1}{q^2}-1\right)\frac{t^2}{2!}\right)\lambda,\nonumber\\
&&\vdots\nonumber\\
\phi_k (t)&=&  \left( 1 + \sum_{m=1}^k \frac{t^m}{m!}\prod_{j=1}^m\left(\frac{1}{q^j}-1\right)\right)\lambda.\nonumber\\	
&&\text{As} \quad k\rightarrow\infty,\quad \phi_k (t)\rightarrow y(t)\nonumber\\
y(t)&=& \left( 1 + \sum_{m=1}^\infty \frac{t^m}{m!}\prod_{j=1}^m\left(\frac{1}{q^j}-1\right)\right)\lambda.\nonumber
\end{eqnarray}
If we define $\prod_{j=1}^m\left(\frac{1}{q^j}-1\right)=1$, for $m=0$,
then

\begin{equation}
y(t)= \left( \sum_{m=0}^\infty \frac{t^m}{m!}\prod_{j=1}^m\left(\frac{1}{q^j}-1\right)\right)\lambda.\label{a5}
\end{equation}
\begin{thm}
	For $q>1$, the  power series (\ref{a5})	is  convergent for $t\in{R}$.
\end{thm}

\begin{cor}
	The power series (\ref{a5}) is  absolutely convergent for all $t$ and hence it is uniformly convergent on any compact interval on ${R}$.
\end{cor}
\begin{thm}
The zero solution of ({\ref{a1}}) is asymptotically stable.
\end{thm}
\section{Fractional order differential equations with proportional delay}
Consider the initial value problem (IVP)
\begin{eqnarray}
D^{\alpha}y(t)&=&f(t, {y}(t), {y}(qt)), 0<\alpha\leq 1, 0<q<1\nonumber\\
y(0)&=&y_0,\label{8}
\end{eqnarray}
where $D^{\alpha}$ denotes Caputo fractional derivative and $f$ is a continuous function defined on the   rectangle
\begin{equation}
R=\{|t|\leq a,|y(t)-y_0|\leq b,|y(qt)-y_0|\leq b, a>0,\, b>0\}\nonumber.
\end{equation}
\begin{thm}
A function $\phi$ is a solution of the IVP (\ref{8}) on an interval $I$ if and only if it is a solution of the integral equation 
\begin{equation}
y(t)=y_0+\int_{0}^{t} \frac{(t-x)^{\alpha-1}}{\Gamma(\alpha)}f(x, {y}(x), {y}(qx))dx \quad \text{on}\quad I. \label{9}
\end{equation}
\end{thm}

\begin{thm}\label{thm5}
Let f is continuous and $|f|\leq M$ on $R$. The successive approximation 
\begin{eqnarray}
\phi_{k+1}(t)&=&y_0\nonumber\\
\phi_{k+1}(t)&=&y_0+\int_{0}^{t} \frac{(t-x)^{\alpha-1}}{\Gamma(\alpha)}f(x, {\phi_k}(x), {\phi_k}(qx))dx.\quad k=0,1,2,\cdots.\label{10}
\end{eqnarray}

 exist and continuous on the interval $I=[-\zeta,\zeta]$, where $\zeta=\text{min}\left\{a,(\frac{\Gamma(\alpha+1)b}{M})^\frac{1}{\alpha}\right\}$.  
	If $t\in I$  then $\left(t,y(t),y(qt)\right)\in R$ and $|\phi_{k}(t)-y_0|\leq M\frac{|t|^\alpha}{\Gamma(\alpha+1)}$, $|\phi_{k}(qt)-y_0|\leq M\frac{|t|^\alpha}{\Gamma(\alpha+1)}$.
\end{thm}
\begin{thm}(Existence Theorem)
Let f is continuous and $|f|\leq M$ on the rectangle
\begin{equation}
R=\{|t|\leq a,|y(t)-y_0|\leq b,|y(qt)-y_0|\leq b, a>0,\, b>0\}\nonumber.
\end{equation}
Suppose $f$ satisfies Lipschitz condition in second and third variable with  Lipschitz constants $L_1$ and $L_2$ such that 
\begin{equation}
|f(t, {y_1}(t), {y_1}(qt))-f(t, {y_2}(t), {y_2}(qt))|\leq L_1|y_1(t)-y_2(t)|+L_1|y_1(qt)-y_2(qt)|\nonumber.
\end{equation}
Then the successive approximations (\ref{10}) converges on the interval  $I=[-\zeta,\zeta]$, where $\zeta=\text{min}\left\{a,(\frac{\Gamma(\alpha+1)b}{M})^\frac{1}{\alpha}\right\}$ to a solution $\phi$ of the IVP  (\ref{8}) on $I$.
\end{thm}

\subsection{Series Solution of Fractional Order Pantograph Equation}	
Consider the  fractional order pantograph equation as :

	\begin{equation}
	D^\alpha y(t) = ay(t)+ by(q t),\quad y(0)= 1, \label{19}
	\end{equation}
	where  $0<\alpha \leq 1$, $ 0<q<1$, $a, b \in{R}$.\\
	
	The solution of (\ref{19}) using successive approximation is
	
	\begin{equation}
	y(t) =  \sum_{m=0}^\infty \frac{t^{\alpha m}}{\Gamma{(\alpha m + 1)}}\prod_{j=0}^{m-1}\left(a+bq^{\alpha j}\right).\label{20}
	\end{equation}
	\begin{thm}
	If $0<q<1$, then the power series (\ref{20}) is convergent for all finite values of $t$.
	\end{thm}
\begin{thm}
If $ 0<q<1$, $a,b\ge 0$, then 
\begin{equation*}
 E_\alpha{(at^\alpha)}\le y(t) =  \sum_{m=0}^\infty \frac{t^{\alpha m}}{\Gamma{(\alpha m + 1)}}\prod_{j=0}^{m-1}\left(a+bq^{\alpha j}\right)\le E_\alpha{((a+b)t^\alpha)},\quad 0 \le t < \infty.
\end{equation*}
\end{thm}	

\subsection{Series Solution of Fractional Order Ambartsumian Equation}
Consider the  fractional order Ambartsumian equation as:
\begin{equation}
	D^\alpha y(t) = -y(t)+\frac{1}{q}y\left(\frac{t}{q}\right),\quad y(0)= 1\label{24}
\end{equation}
where $q>1$ and is constant for the given model.

	The solution of (\ref{24}) using successive approximation is
	
	\begin{equation}
	y(t) =  \sum_{m=0}^\infty \frac{t^{\alpha m}}{\Gamma{(\alpha m + 1)}}\prod_{j=0}^{m-1}\left(\frac{1}{q^{1+\alpha j}}-1\right).\label{25}
	\end{equation}
	\begin{thm}
	If $q>1$, then the power series (\ref{25}) is convergent for all finite values of $t$.
	\end{thm}

\section{System of fractional order differential equations with proportional delay}
Consider the initial value problem (IVP)
\begin{eqnarray}
D^{\alpha_i}y_i(t)&=&f_i(t, \bar{y}(t), \bar{y}(qt)), 0<\alpha_i\leq 1, 0<q<1\nonumber\\
y_i(0)&=&^iy_0,\quad 1\leq i \leq n,\label{11}
\end{eqnarray}
where $D^{\alpha_i}$ denotes Caputo fractional derivative, $\bar{y}(t)=(y_1(t),y_2(t)\cdots,y_n(t))$,\\ $\bar{y}(qt)=(y_1(qt),y_2(qt)\cdots,y_n(qt))$ and $f=(f_1,f_2\cdots,f_n)$ is a continuous function defined on the rectangle
\begin{equation}
R=\{|t|\leq a,|y_i(t)-^iy_0|\leq b_i,|y_i(qt)-^iy_0|\leq b_i, a>0,\, b_i>0, 1\leq i \leq n\}\nonumber.
\end{equation}

\begin{thm}
A function $\bar\phi$ is a solution of the IVP (\ref{11}) on an interval $I$ if and only if it is a solution of the integral equation 
\begin{equation}
y_i(t)=^iy_0+\int_{0}^{t} \frac{(t-x)^{\alpha_i-1}}{\Gamma(\alpha_i)}f(x, \bar{y}(x), \bar{y}(qx))dx \quad \text{on}\quad I, 
\end{equation}
where $\bar{\phi}_m=\left(^1\phi_m, ^2\phi_m,\cdots, ^n\phi_m\right)$
\end{thm}
\begin{thm}
	Let $||f||=M$ on rectangle R. The successive approximation 
	\begin{eqnarray}
	^i\phi_{0}(t)&=&^iy_0, \quad i=0,1,2,\cdots.\nonumber\\
	^i\phi_{k+1}(t)&=&y_0+\int_{0}^{t} \frac{(t-x)^{\alpha_i-1}}{\Gamma(\alpha_i)}f(x, {\bar\phi_k}(x), {\bar\phi_k}(qx))dx.\quad k=0,1,2,\cdots.\label{12}
	\end{eqnarray}
	 exist and continuous on the interval $I=[-\zeta,\zeta]$, where 
	\begin{equation*}
	\zeta=\text{min}\left\{a, \left(\frac{\Gamma(\alpha_1+1)b_1}{M}\right)^{\frac{1}{\alpha_1}},\cdots, \left(\frac{\Gamma(\alpha_n+1)b_n}{M}\right)^{\frac{1}{\alpha_n}}\right\}.
	\end{equation*}
	If $t$ is in interval $I$ then $\left(t,\bar{y}_{m}(t),\bar{y}_{m}(qt)\right)$ is in rectangle R and $||\bar{y}_{m}(t)-\bar{y}(0)||\leq M\sum_{i=1}^{m}\frac{|t|^{\alpha_i}}{\Gamma(\alpha_i+1)}$, $||\bar{y}_{m}(qt)-\bar{y}(0)||\leq M\sum_{i=1}^{m}\frac{|t|^{\alpha_i}}{\Gamma(\alpha_i+1)}\, \forall m$.
\end{thm}
\begin{thm}
	Let $f$ be a continuous function defined on the rectangle 
	\begin{equation}
R=\{|t|\leq a,|y_i(t)-^iy_0|\leq b_i,|y_i(qt)-^iy_0|\leq b_i, a>0,\, b_i>0, 1\leq i \leq n\}\nonumber.
	\end{equation}
	Suppose $f$ satisfies Lipschitz condition in second and third variable with  Lipschitz constants $L_1$ and $L_2$ such that 
$
	|f(t, \bar{y}(t), \bar{y}(qt))-f(t, \bar{y}(t), \bar{y}(qt))|\leq L_1|\bar y_1(t)-\bar y_2(t)|+L_1|\bar y_1(qt)-\bar y_2(qt)|\nonumber.
	$
	Then the successive approximations (\ref{12}) converges on the interval  $I=[-\zeta,\zeta]$, where $\zeta=\text{min}\left\{a, \left(\frac{\Gamma(\alpha_1+1)b_1}{M}\right)^{\frac{1}{\alpha_1}},\cdots, \left(\frac{\Gamma(\alpha_n+1)b_n}{M}\right)^{\frac{1}{\alpha_n}}\right\}
		$ to a solution of the $\phi$ of the IVP  (\ref{11}) on $I$.
	\end{thm}
	
\subsection{System of Fractional Order Pantograph Equation}
Consider the system of fractional order pantograph equation
\begin{equation}
D^\alpha y(t) = Ay(t)+ By(q t),\quad y(0)= y_0,  \quad 0< \alpha \le 1\label{26}
\end{equation}
where $0<q<1$, $A=\left(a_{ij}\right)_{n\times n}$, $B=\left(b_{ij}\right)_{n\times n}$ and $y=[y_1,y_2,\cdots,y_n]^T$
	The solution of (\ref{16}) using successive approximation is
\begin{eqnarray}
 y(t)&=& \left[\sum_{k=0}^\infty\prod_{j=1}^{k}(A+Bq^{-(k-j)\alpha})\frac{ t^{k\alpha}}{\Gamma (k\alpha+1)}\right]\lambda.\label{27}
\end{eqnarray}
\begin{thm}
	 For $0<q<1$, the  power series (\ref{27})  is  convergent for $t\in{R}$.
\end{thm}

\subsection{ System of Fractional Order Ambartsumian Equations}
In this section, we generalize the Ambartsumian equation (\ref{1}) to the system of fractional order Ambartsumian equations \cite{patade2020series} as: 

\begin{equation}
D^\alpha y(t) = -Iy(t)+By\left(\frac{t}{q}\right),  \quad y(0)=\lambda, \quad 0< \alpha \le 1, \label{2}
\end{equation}
where  $D^\alpha$ denotes Caputo fractional derivative, $I$ is the identity matrix of order $n$,  $1<q$, \\
$y=$$
\begin{bmatrix}
y_1\\
y_2\\
\vdots\\
y_n
\end{bmatrix}
$, 
$\lambda=$
$\begin{bmatrix}
\lambda_1\\
\lambda_2\\
\vdots\\
\lambda_n
\end{bmatrix}$
and 
 $B=$$
\begin{bmatrix}
\frac{1}{q}& 0 & 0 & \cdots & 0\\
0& \frac{1}{q} & 0 & \cdots & 0\\
\vdots& \vdots & \vdots & \ddots & 0\\
0& 0 & 0 & \cdots & \frac{1}{q}
\end{bmatrix}_{n\times n}.
$\\

 Applying SAM to the initial value problem (\ref{2}), we have
 \begin{equation}
  y(t) = y(0) -I J^\alpha y(t)+ B J^\alpha y\left(\frac{t}{q}\right)\label{3}.
 \end{equation}
 Suppose $\phi_k(t)$ be the $k$th approximate solution, where the initial approximate solution is
 taken as
\begin{equation}
\phi_0(t)=\lambda.
\end{equation}

For $k \ge 1$, the recurrent formula as below:
 \begin{equation}
\phi_k (t) =\lambda -I J^\alpha \phi_{k-1} (t)+ B J^\alpha  \phi_{k-1}\left(\frac{t}{q}\right)\label{4}.
\end{equation}
From the recurrent formula, we have 
\begin{eqnarray}
	\phi_1 (t) &=& \lambda -I J^\alpha \phi_{0} (t)+ B J^\alpha  \phi_{0}\left(\frac{t}{q}\right)\nonumber\\
	&=&\lambda -I \frac{\lambda t^\alpha}{\Gamma (\alpha+1)}+ B \frac{\lambda t^\alpha}{\Gamma (\alpha+1)}\nonumber\\
	&=&\left(I +(-I +B)\frac{ t^\alpha}{\Gamma (\alpha+1)}\right)\lambda,\nonumber\\
\phi_2 (t) &=& \lambda -I J^\alpha \phi_{1} (t)+ B J^\alpha  \phi_{1}\left(\frac{t}{q}\right)\nonumber\\
&=&\lambda -I J^\alpha \left[\left(I +(-I +B)\frac{t^\alpha}{\Gamma (\alpha+1)}\right)\lambda\right]+ B J^\alpha  \left(I +(-I +B)\frac{ q^{-\alpha}t^\alpha}{\Gamma (\alpha+1)}\right)\lambda\nonumber\\
&=&\lambda -I \left[\frac{\lambda t^\alpha}{\Gamma (\alpha+1)} +(-I +B)\frac{\lambda t^{2\alpha}}{\Gamma (2\alpha+1)}\right]+ B   \left[\frac{\lambda t^\alpha}{\Gamma (\alpha+1)} +(-I +B)\frac{\lambda q^{-\alpha} t^{2\alpha}}{\Gamma (2\alpha+1)}\right]\nonumber\\
&=&\left[I +(-I +B)\frac{ t^\alpha}{\Gamma (\alpha+1)} +  (-I +Bq^{-\alpha})(-I +B)\frac{ t^{2\alpha}}{\Gamma (2\alpha+1)}\right]\lambda,\nonumber
\end{eqnarray}

\begin{eqnarray}
\phi_3 (t)&=&\left[I +(-I +B)\frac{ t^\alpha}{\Gamma (\alpha+1)} +  (-I +Bq^{-\alpha})(-I +B)\frac{ t^{2\alpha}}{\Gamma (2\alpha+1)}\right.\nonumber\\
&&\left.+(-I +Bq^{-2\alpha})(-I +Bq^{-\alpha})(-I +B)\frac{ t^{3\alpha}}{\Gamma (3\alpha+1)}\right]\lambda,\nonumber\\
&&\cdots,\nonumber\\
\phi_k (t)&=& \left[I + \sum_{m=1}^k\prod_{j=1}^{m}(-I+Bq^{-(m-j)\alpha})\frac{ t^{m\alpha}}{\Gamma (m\alpha+1)}\right]\lambda\nonumber
\end{eqnarray}
As $k\rightarrow\infty$,$\quad \phi_k (t)\rightarrow y(t)$
\begin{eqnarray}
y(t)&=& \left[I + \sum_{k=1}^\infty\prod_{j=1}^{k}(-I+Bq^{-(k-j)\alpha})\frac{ t^{k\alpha}}{\Gamma (k\alpha+1)}\right]\lambda.\nonumber
\end{eqnarray}
If we set $\prod_{j=1}^{k}(-I+Bq^{(k-j)\alpha})=I$, for $k=0$,
then

\begin{eqnarray}
 y(t)&=& \left[\sum_{k=0}^\infty\prod_{j=1}^{k}(-I+Bq^{-(k-j)\alpha})\frac{ t^{k\alpha}}{\Gamma (k\alpha+1)}\right]\lambda.\label{5}
\end{eqnarray}

\begin{thm}
	 For $q>1$, the  power series 
	\begin{eqnarray}
	y(t)&=& \left[\sum_{k=0}^\infty\prod_{j=1}^{k}(-I+Bq^{-(k-j)\alpha})\frac{ t^{k\alpha}}{\Gamma (k\alpha+1)}\right]\lambda\nonumber
	\end{eqnarray}
 is  convergent for $t\in{R}$.
\end{thm}
\textit{\textbf{Proof:}}$\quad$
Result follows immediately by ratio test \cite{apostol1958mathematical}.
\section{Illustrative Examples}\label{exampl}
\textbf{Example 1.} Consider the  non-linear differential equations with proportional delay \cite{evans2005adomian,rangkuti2012exact,alomari2009solution,ratib2013optimal}
\begin{equation}
\frac{dy(t)}{dt} = 1 - 2 y^2\left(\frac{t}{2}\right),\quad y(0) = 0. \label{21}
\end{equation}

The corresponding integral equation is
\begin{equation}
y(t)= \int_{0}^{t}\left(1-2u^2\left(\frac{t}{2}\right)\right) dx.
\end{equation}
By using successive approximation method  (\ref{22}), we obtain
\begin{eqnarray}
\phi_0(t) &=& 0,\nonumber\\
\phi_1(t) &=&t,\nonumber\\
\phi_2(t) &=& t-\frac{t^3}{6},\nonumber\\
\phi_3(t) &=& t-\frac{t^3}{6}+\frac{t^5}{120}-\frac{t^7}{8064},\nonumber\\
\phi_4(t) &=& t-\frac{t^3}{6}+\frac{t^5}{120}-\frac{t^7}{5040}+\frac{61 t^9}{23224320}-\frac{67 t^{11}}{3406233600}+\frac{t^{13}}{12881756160}-\frac{t^{15}}{7990652436480},\nonumber\\
\phi_5(t) &=& t-\frac{t^3}{6}+\frac{t^5}{120}-\frac{t^7}{5040}+\frac{61 t^9}{23224320}-\cdots-\frac{t^{31}}{1062664199886151693758358595882188800},\nonumber
\end{eqnarray}
and so on.\\
The exact solution of Eq.(\ref{21}) is \, $y(t) = \sin t$.\\
The 5-term solutions of Eq.(\ref{21}) using Adomian decomposition method (ADM) \cite{evans2005adomian}, variational iteration method (VIM) \cite{rangkuti2012exact}, homotopy analysis method (HAM) \cite{alomari2009solution}, optimal homotopy asymptotic method (OHAM) \cite{ratib2013optimal}   are same and is given by
\begin{eqnarray}
y(t) &=& t-\frac{t^3}{6}+\frac{t^5}{120}-\frac{t^7}{5040}+\frac{t^9}{362880}-\frac{t^{11}}{39916800}+\frac{t^{13}}{6227020800}\nonumber\\
&&-\frac{t^{15}}{1307674368000}+\frac{t^{17}}{355687428096000}.\label{n1.4}
\end{eqnarray}
The 4-term OHAM solution \cite{ratib2013optimal} of Eq.(\ref{21}) is
\begin{eqnarray}
y(t) &=& t-0.166665 t^3+0.00832857 t^5 - 0.000192105 t^7.
\end{eqnarray}
We compare $5^{th}$ approximation solution (SAM) and 5-term solutions ( ADM, VIM, HAM) with  exact solution   in Fig.(1) and $4^{th}$ approximation solution (SAM) with 4-term solution (OHAM) in Fig.(2). The absolute errors in computation are shown in Figs.(3)-(4). It can be observed that SAM solution is better than the solution obtained by using other  methods.

\begin{tabular}{c}
\includegraphics[scale=0.5]{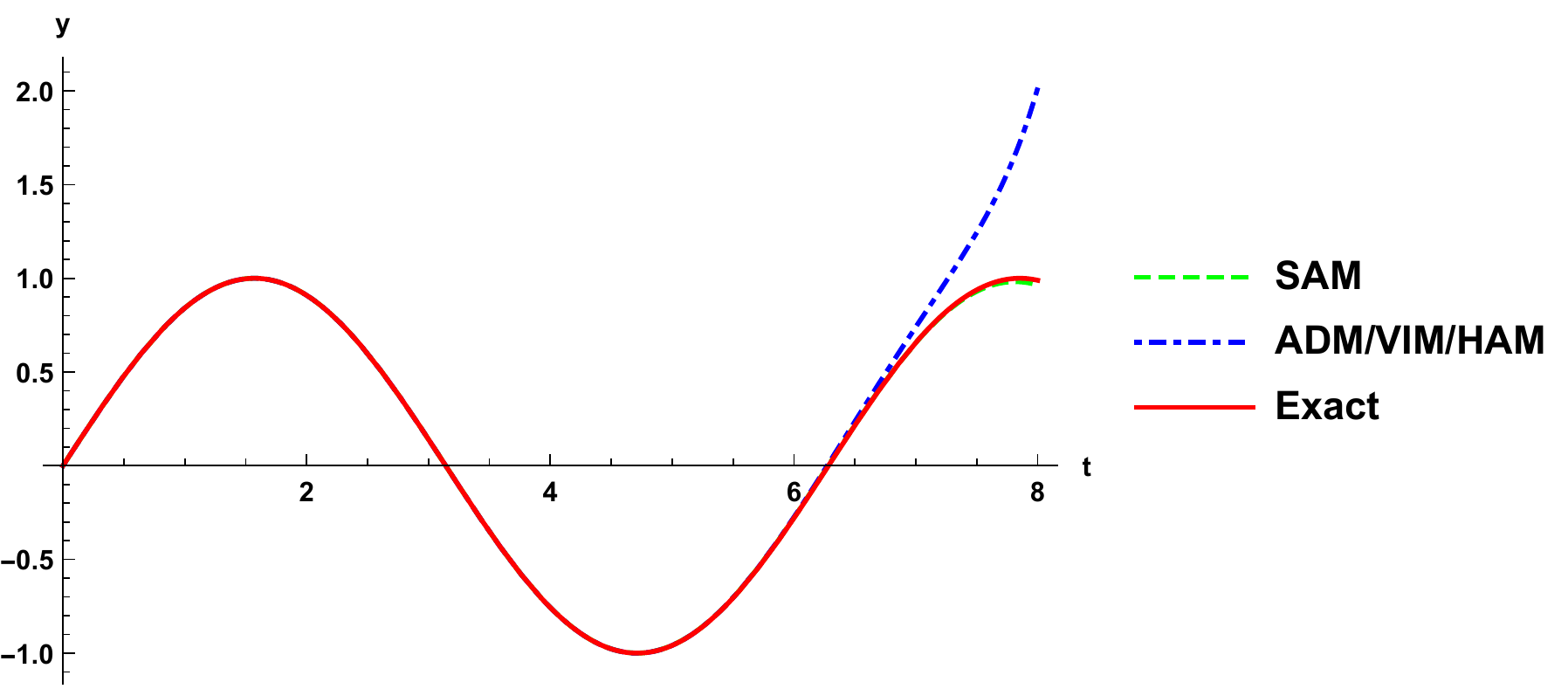}\\
Fig.1:Comparison of SAM, ADM/VIM/HAM solutions with exact solution of Eq.(\ref{21}).\\
\end{tabular}

\begin{tabular}{c}
\includegraphics[scale=0.5]{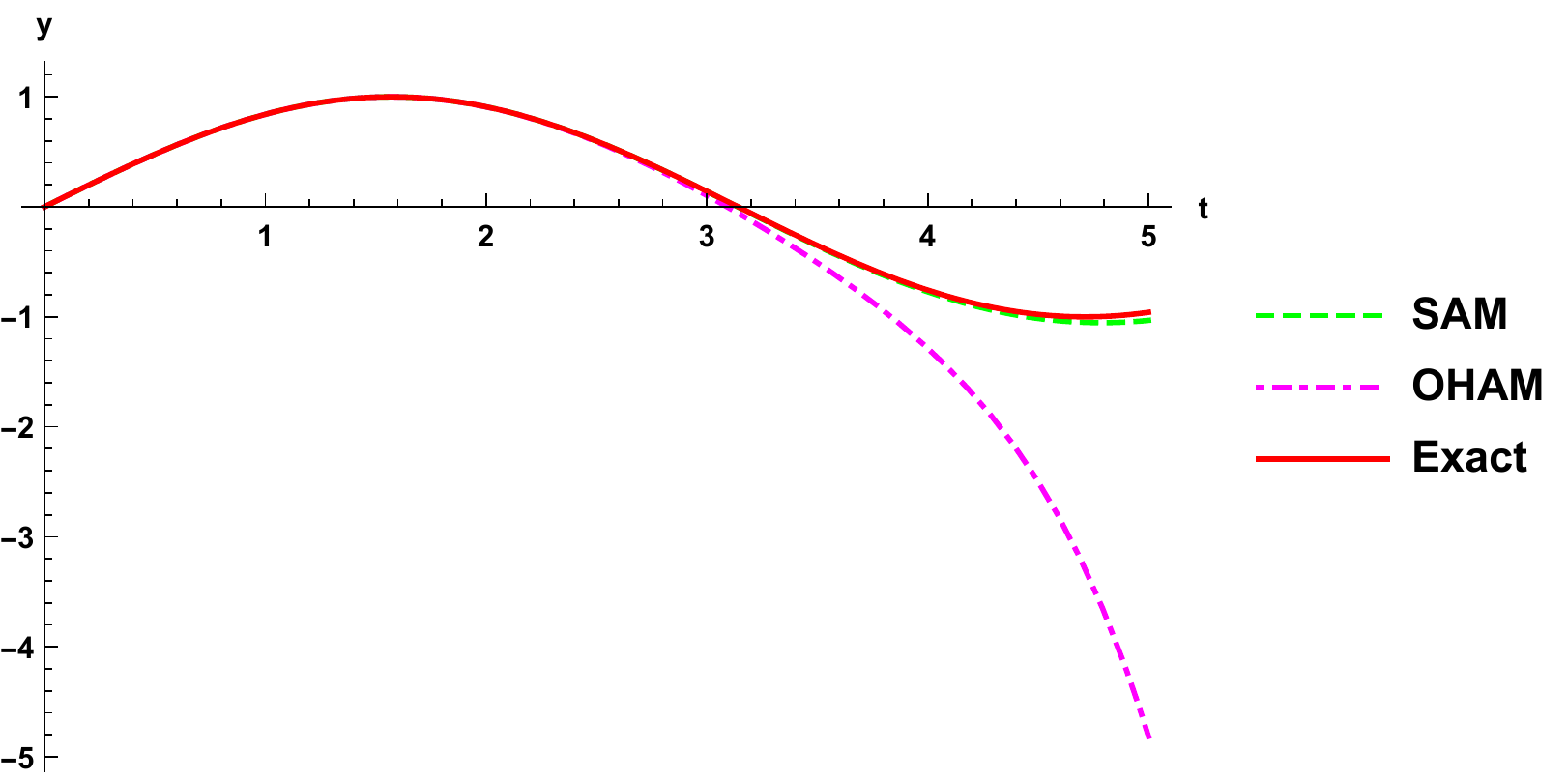}\\
Fig.2:Comparison of SAM, OHAM solutions with exact solution of Eq.(\ref{21}).\\
\end{tabular}

Remark: The ADM/VIM/HAM and OHAM solutions  given in \cite{evans2005adomian,rangkuti2012exact,alomari2009solution,ratib2013optimal} are considered only for the interval $[0,1]$ . Here we have successfully extended the solution using SAM in intervals $[0,8]$.

\begin{tabular}{c}
\includegraphics[scale=0.5]{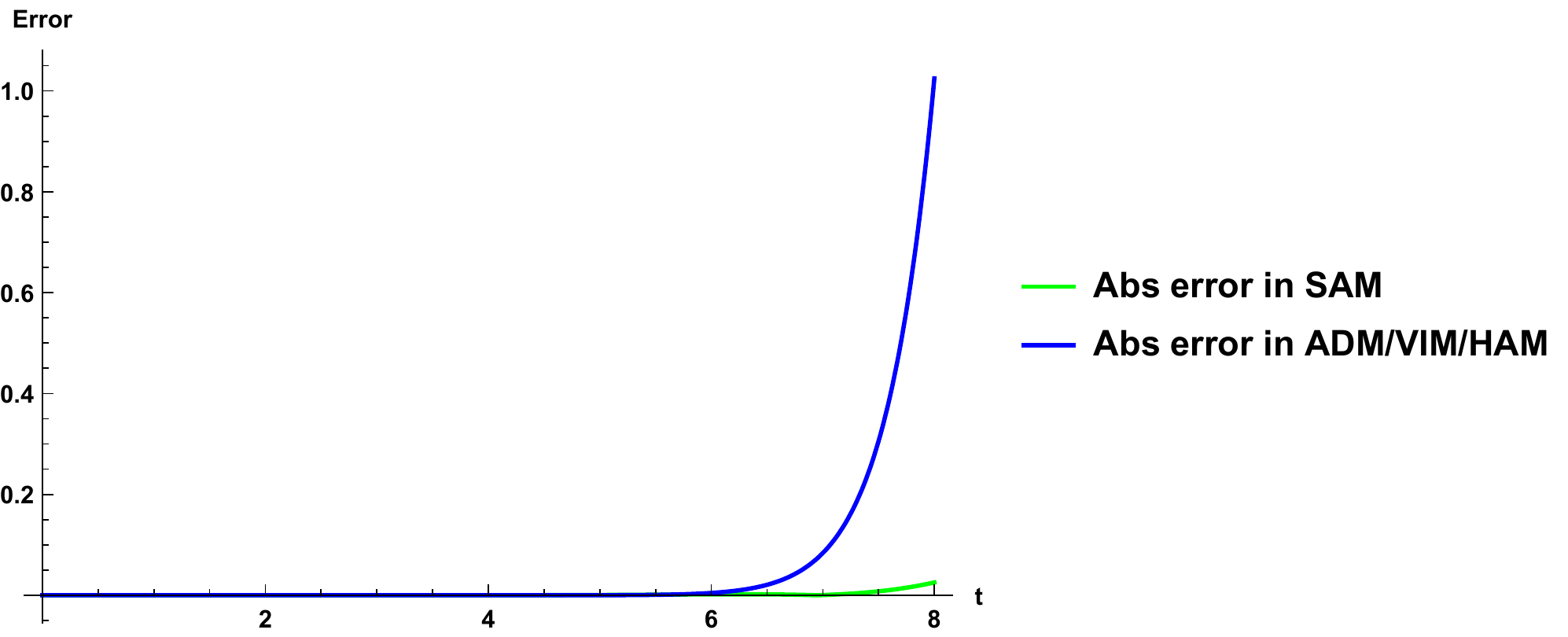}\\
Fig.3:Comparison of absolute errors in SAM and ADM/VIM/HAM solutions.\\
\end{tabular}

\begin{tabular}{c}
\includegraphics[scale=0.5]{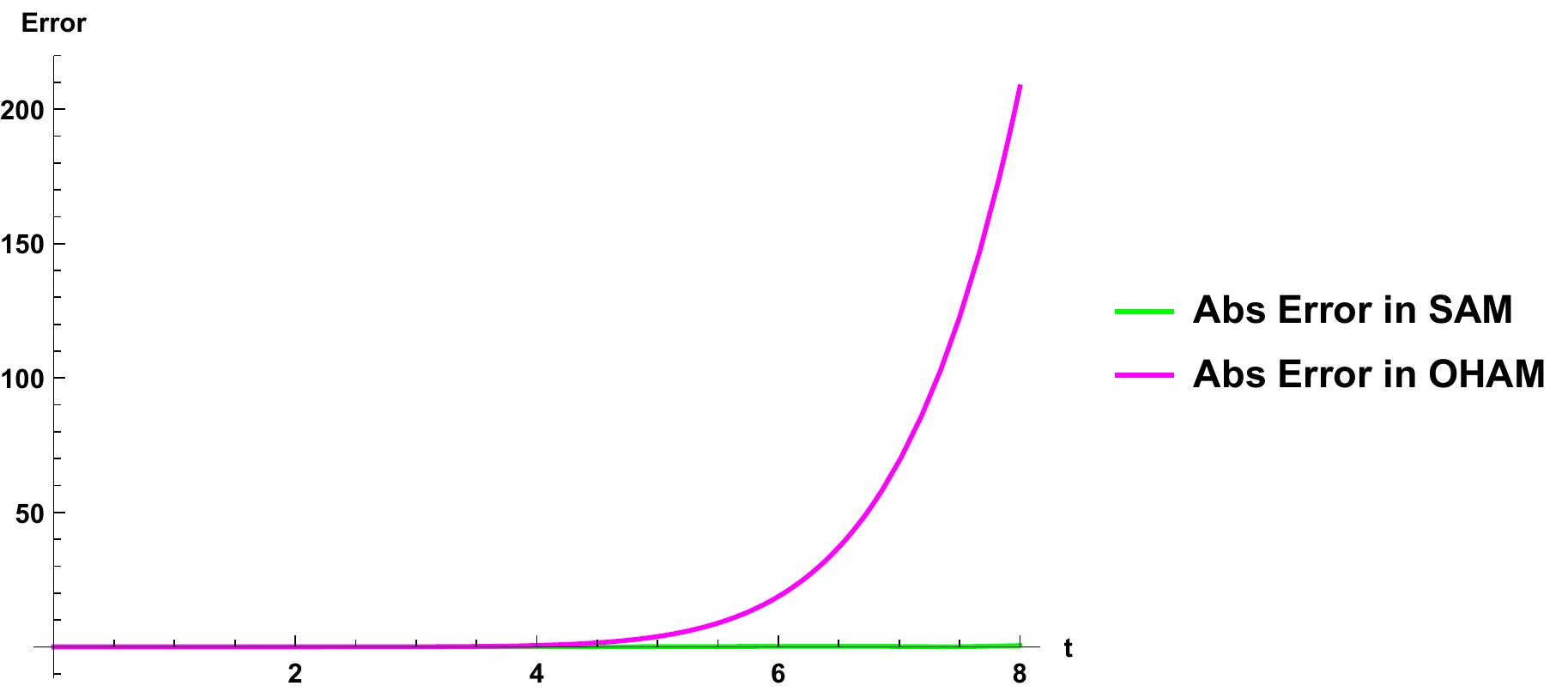}\\
Fig.4:Comparison of absolute errors in SAM and OHAM solutions.\\
\end{tabular}

\section{Conclusions}\label{concl}
  In this paper, we solved non-linear differential equations with proportional delay using the successive approximation method (SAM). The existence, uniqueness, and stability theorems for differential equations with proportional delay are presented. The convergence results are derived by using the Lipschitz condition. A generalization of fractional order and a system of fractional order cases are also presented. The series solution of the pantograph equation and Ambersumian equation is obtained using SAM. Finally, we illustrated the effectiveness of SAM through an example.

\end{document}